\documentclass[10pt,a4paper]{amsart}
\usepackage[utf8]{inputenc}
\usepackage{amsmath, amsrefs}
\usepackage{amsfonts}
\usepackage{amssymb,color}
\usepackage{pstricks}
\usepackage[all]{xy}
\usepackage{amsfonts}
\usepackage[english]{babel}

\newtheorem{ej}{\textbf{Example} \rm }
\theoremstyle{remark}
\newtheorem{defi}{Definition}[section]
\newtheorem{remark}{Remark}[section]
\theoremstyle{plain}

\newtheorem{prop}{Proposition}[section]
\newtheorem{teo}{Theorem}
\newtheorem{lema}{ \textbf{Lemma} \rm}
\providecommand{\abs}[1]{\lvert#1\rvert}
\DeclareMathOperator{\iti}{{\mathcal I}}
\DeclareMathOperator{\Z}{{\mathbb Z}}

\begin{document}

\title{Nonstandard Expansiveness}

\author[L. Ferrari]{Luis Ferrari}
\address{Departamento de Matemática y Aplicaciones, Cure, Universidad de la República, Maldonado, Uruguay}
\email{luisrferrari@gmail.com}
\date{\today}

\begin{abstract}
Let $(X,d)$ be a  metric space and $f: X \rightarrow X$ be a homeomorphism. We say that a dynamical system $(X,f)$ is 
\emph{expansive}, with constant of expansivity $c \in \mathbb{R{^+}}$, if for all $x,y \in X$ , $x \neq y$, exists $n \in \mathbb{Z}$, such that $d(f^n(x), f^n(y)) >c$. In this paper we will use the theory of Nonstandard Analysis to study a subfamily of these dynamics, which verify that for all $x,y \in X$, if $x\neq y$ then the set  $\lbrace n \in \mathbb{Z} : d(f^n(x), f^n(y) > c \rbrace$ is infinite.\\

\end{abstract}
\maketitle

\section{Introduction}

\vspace{0.3cm}

Expansive dynamics are an important type of dynamical system. Furthermore, these dynamics have interesting connections with topology and other dynamical concepts. An example of this can be seen in the relationship between the topology of the space and the asymptotic  behavior of the dynamics. We know by \cite{utz1950unstable} that every expansive homeomorphism on an infinite compact metric space has asymptotic points; that is, exists $x\neq y$ such that $\lim_{n \rightarrow + \infty} d(f^n(x), f^n(y)) = 0$ (positive asymptotics) or  negative asymptotics ($n\to -\infty$). Dynamic systems without doubly asympotic points(positive and negative asymptotic) are a peculiar type of dynamic system. All of the examples that we know of expansive homeomorphisms on compact metric spaces that are not totally disconnected have
asymptotic points. 
By \cite{artigue2020subshift}, we know that there are infinite dynamics of this type such that they do not conjugate with each other but in the form of a subshift, and therefore on a totally disconnected space.\\
Groisman and da Silva \cite{jorgesamuel} introduce the notion of \emph{freely expansive dynamic}. Let $X$ be a compact metric space and a $f: X \rightarrow X$ homeomorphism, then the dynamical system $(X,f)$ is freely expansive if there is some $c >  0$ such that for every  distinct points $x,y \in X$ there is some free ultrafilter $p$ over $\mathbb{N}$ such that either $d(f^p(x), f^p(y)) > c$ or $d(f^{-p}(x), f^{-p}(y)) > c$, where $f^p$ is the $p$-iterate of $f$ for the future, and $f^{-p}$ is the $p$-iterate of $f$ for the past, the existence of $p$-iterate is guaranteed by the compactness of the space. In their notes, the authors demonstrate that these dynamics are equivalent of dynamical systems such that for all different $x,y \in X$ the set  $\lbrace n \in \mathbb{Z} : d(f^n(x), f^n(y) > c \rbrace$ is infinite when $X$ is compact. Another important  fact is that freely expansive homeomorphism are equivalent to an expansive homeomorphism without doubly asymptotic points.\\\\
\emph{Why use nonstandard analysis?}\\
Suppose that $(X,f)$ is a dynamical system such that  for all $x,y \in X$, $x \neq y$ the set $\lbrace n \in \mathbb{Z} : d(f^n(x), f^n(y) > c \rbrace$ is infinite, and suppose without lost of generality that is in a succession of positive integers $(\lambda_n)_{n \in \mathbb{Z}^{+}}$; that is, \\\\
$d(f^{\lambda_1}(x), f^{\lambda_1}(y)) > c, d(f^{\lambda_2}(x), f^{\lambda_2}(y)) > c, \ldots, d(f^{\lambda_n}(x), f^{\lambda_n}(y)) > c, \ldots$.\\
Consider the  succession of constant succession $(x)_{n \in \mathbb{Z}^{+}}, (y)_{n \in \mathbb{Z}^{+}} (c)_{n \in \mathbb{Z}^{+}}$, and the succession of the moments of separations $(\lambda_n)_{n \in \mathbb{Z}^{+}}$.  In the construction of the nonstandard extension of a set via a free ultrafilter \cite{loebwolf}, this succession corresponds  (respectively) to the points $\mathsf{x}$, $\mathsf{y}$ in ${^*}X$; nonstandard extension of $X$, $\mathsf{c}$  in ${^*}\mathbb{R}$; nonstandard extension of $\mathbb{R}$, and  $\mathsf{\lambda}$ in ${^*}\mathbb{Z}$; and nonstandard extension of $\mathbb{Z}$,  such that\\
$${^*}d({^*}f^{\mathsf{\lambda}}(\mathsf{x}),{^*}f^{\mathsf{\lambda}}(\mathsf{y})) > \mathsf{c}$$.\\
Where ${^*}d : {^*}X \times {^*}X \rightarrow {^*}\mathbb{R}$ is the nonstandard extension of the metric $d$, and ${^*}f : {^*}X  \rightarrow {^*}X$ the nonstandard extension of the homeomorphism $f$. Note that we never use the compactness of $X$.\\
We will use nonstandard analysis methods to study these dynamics. We will see that nonstandard analysis is particularly useful to study these dynamics where asymptotic behavior is crucial. Furthermore, this approach allows us to introduce new concepts and questions.

\section{Preliminaries of expansive systems}

\begin{defi}\label{defi1} Let $X$ be a metric space and $f : X \rightarrow X$ be a homeomorphism. We say that $f$ is \emph{expansive} if exists $c \in \mathbb{R^{+}}$ such that for all $x,y \in X$, $x \neq y$, exists $n \in \mathbb{Z}$, such that $d(f^n(x), f^n(y)) > c$.
\end{defi}

\begin{defi} Let $X$ and $Y$ be compact metric spaces and $f : X \rightarrow X$, $g : Y \rightarrow Y$ be two homeomorphisms. We say that the dynamics $(X, f)$, $(Y, g)$  are \emph{conjugated} if  $\varphi : X \rightarrow Y$ homeomorphisms exist such that $\varphi^{-1} .g.\varphi = f$.
\end{defi}

\begin{prop} Let $X$ and $Y$ be  compact metric spaces, and $(X,f)$, $(Y, g)$ be conjugated dynamics. If  $f : X \rightarrow X$ is expansive,  then  $g$ is also expansive.
\end{prop}
\begin{proof}
Suppose that $f$ is expansive with  constant of expansivity $\alpha$. We will prove that $g=\varphi.f.\varphi^{-1}$ is expansive; that is, will find a $c\in \mathbb{R}$, $c >0$ such that for all  $a,b \in Y$, different, $d(g^n(a),g^n(b)) < c$ is verified for all $n\in \mathbb{Z}$, then $a=b$.\\
$\varphi^{-1}$ is uniform continuity because $X$ is compact, then $c\in \mathbb{R}$, $c>0$ exists such that if $d(y,y') < c$, then  $d(\varphi^{-1}(y),\varphi^{-1}(y')) < \alpha$. Suppose that for $a, b \in Y$ we have $d(g^n(a),g^n(b)) < c$ for all $n \in \mathbb{Z}$. Then,  $d(\varphi^{-1}(g^n(a),\varphi^{-1}(g^n(b))) < \alpha$ for all $n \in \mathbb{Z}$, but  $a= \varphi(x)$ y $b = \varphi (x')$ for some $x,x' \in X$. Then, $d(\varphi^{-1}(g^n(\varphi(x)),\varphi^{-1}(g^n(\varphi(x')))) < \alpha$, but $\varphi^{-1}(g^n(\varphi(z)) = f^n(z)$ for all $z \in X$. Then,  $d(f^n(x),f^n(x')) <\alpha$ for all  $n\in \mathbb{Z}$. If $f$ is $\alpha$-expansive, then we have  $x=x'$ and then $a=b$.
\end{proof}

\begin{prop} Let $X$ be a compact metric space and $f : X \rightarrow X$ be expansive homeomorphisms; then for all  $k \in \mathbb{N}$, $f^k$ is expansive.
\end{prop}
\begin{proof} Let $c$ be the constant of expansivity of $f$, for the uniformly continuous we know that $\delta > 0$ exists such that  $d(x,y) < \delta$, then $d(f^i(x),f^i(y)) < c$ for all $i= 1, \ldots k$. Suppose that $f^k$ is not  expansive, then $x_0,y_0 \in X$  $m\in \mathbb{Z}$ exists such that $d(f^{mk}(x_0),f^{mk}(y_0)) < \delta$. Let $n \in \mathbb{Z}$, then for the Euclid theorem exists $m, i \in \mathbb{Z}$ , with  $0\leq i < k$ such that $n=mk + i$, then  $d(f^n(x_0), f^n(y_0)) = d(f^{mk + i}(x_0),f^{mk + i}(y_0)) = d(f^i(f^{mk}(x_0)),f^i(f^{mk}(y_0))) < c$, then  $f$ is not expansive, which is absurd.
\end{proof}

Uniform expansiveness is a useful concept in the theory of expansive dynamics.
\begin{defi}
Let $X$ be a compact metric space and  $f : X \rightarrow X$ be homeomorphisms. We say that the dynamical system $(X,f)$ is $c$-uniformly expansive if for all $\epsilon > 0$ exists $n_{\epsilon} \in \mathbb{N}$ such that $x,y \in X$, if  $d(x,y) > \epsilon$ then $d(f^i(x),f^i(y)) > c$, for some $i \in \mathbb{Z}$, $\abs{i} < n_{\epsilon}$.
\end{defi}

\begin{prop} 
Let $X$ be a compact metric space and $f : X \rightarrow X$ be homeomorphisms. If  $(X,f)$ is $c$-expansive, then $f$ is $c$-uniformly expansive.
\end{prop}
\begin{proof}
Suppose that  $f$ is not $c$-uniformly expansive, then $\epsilon > 0$ exists such that for all  $n$ exists  $x_n, y_n \in X$, such that  $d(x_n, y_n) > \epsilon$ and $d(f^i(x_n),f^i(y_n)) \leq c$ for all $i \in \mathbb{Z}$, with  $\abs{i} < n$, then how  $X$ is compact. Taking subsuccession if necessary, we can suppose that $\lim_{n \rightarrow + \infty} x_n = x$, $\lim_{n \rightarrow + \infty} y_n = y$ , for some $x,y \in X$. Let $k \in \mathbb{Z}$, then how  $\lim_{n \rightarrow + \infty} x_n = x$, we have  $\lim_{n \rightarrow + \infty} f^k(x_n) = f^k(x)$ y $\lim_{n \rightarrow + \infty} f^k(y_n) = f^k(y)$, taken $\delta > 0$, $n_0$ exists such that $n_0 > \abs{k}$, y $d(f^k(x), f^k(x_n)) < \frac{\delta}{2}$, $d(f^k(y), f^k(y_n)) < \frac{\delta}{2}$ for all $n \geq n_0$, then  \\
$d(f^k(x), f^k(y)) \leq d(f^k(x), f^k(x_{n_0})) + d(f^k(x_{n_0}), f^k(y_{n_0})) + d(f^k(y_{n_0}), f^k(y)) < \frac{\delta}{2} + c + \frac{\delta}{2} = c + \delta$, then \\ $d(f^k(x), f^k(y)) < c + \delta$ for all  $\delta > 0$, then $d(f^k(x), f^k(y)) \leq c$, for all $k \in \mathbb{Z}$, then  $f$ is not $c$-expansive.
\end{proof}

\section{Nonstandard expansiveness}

\begin{defi} Let $f : X \rightarrow X$ be homeomorphisms. We say that $f$ is  \emph{nonstandard expansive} if  $c \in \mathbb{R}$, $c > 0$, exists such that for all $x,y \in X$, in contrast $n \in {^*}\mathbb{Z}_{\infty}$ exists such that ${^*}d(({^*}f)^n(x), ({^*}f)^n(y)) > c$, where ${^*}\mathbb{Z}_{\infty}$ is a set of infinite integers of $\mathbb{Z}$.
\end{defi}

\begin{teo}\label{caractcombnonstandarexpansivos} Let $f : X \rightarrow X$ be homeomorphisms, $f$ is nonstandard expansive with  constant of expansivity $c$ if only for all $x, y \in X$, in contrast the set $\lbrace n \in \mathbb{Z} : d(f^n(x), f^n(y)) > c \rbrace$ is infinite.
\end{teo}
\begin{proof} $\Rightarrow$: Suppose that  $\lbrace n \in \mathbb{Z} : d(f^n(x), f^n(y)) > c \rbrace$  is finite, then exists  $m \in \mathbb{N}$ such that \\
$(\forall n \in \mathbb{N})(n \geq \underline{m} \rightarrow d(f^n(\underline{x}),f^n(\underline{y})) \leq \underline{c})$. Then, for the transfer principle the following formula is true \\ $(\forall n \in {^*}\mathbb{N})(n \geq \underline{m} \rightarrow {^*}d( {^*}f^n(\underline{x}),{^*}f^n(\underline{y})) \leq \underline{c})$; in particular, for all positive integer $n$, $ {^*}d({^*}f^n(x),{^*}f^n(y)) \leq c$, Analogously  $\lbrace n \in \mathbb{Z}_{\leq 0} : d(f^n(x), f^n(y)) > c \rbrace$ and change  $f$ for $f^{-1}$ we have that for all infinite negative integer $n$ $d(f^n(x),f^n(y)) \leq c$, then is not nonstandard expansive.\\\\
$\Leftarrow$: Suppose that  $\lbrace n \in \mathbb{N} : d(f^n(x), f^n(y)) >c \rbrace$ is finite. If it does not work with  $f^{-1}$, then for all $n\in \mathbb{N}$ exists $m >n$ such that $d(f^m(x),f^m(y)) > \underline{c}) $ then there exists a function  $\psi : \mathbb{N} \rightarrow \mathbb{N}$ such that $\forall n \in \mathbb{N} ((\psi(n) >  n) \wedge d(f^{\psi(n)}(\underline{x}),f^{\psi(n)}(\underline{y})) > \underline{c})$. Then, for transfer principle  $\forall n \in {^*}\mathbb{N} (({^*}\psi(n) >  n) \wedge d(f^{{^*}\psi(n)}(\underline{x}),f^{{^*}\psi(n)}(\underline{y})) > \underline{c})$. Then, if $m\in {^*}\mathbb{N}_{\infty}$ \footnote{${^*}\mathbb{N}_{\infty}$ is the set of infinite numbers of ${^*}\mathbb{N}$}. Then,  ${^*}\psi(m) \in  {^*}\mathbb{N}_{\infty}$. Consequently, $f$ is nonstandard expansive.
\end{proof}

\begin{remark} (\emph{Strongly nonstandard expansive?}) A natural question is what happens if in the definition of nonstandard expansive instead of quantifying on $X$ we quantify on ${^*}X$? That is, we could define  $f : X \rightarrow X$ as $c$-strongly nonstandard expansive if only for all $x,y \in {^*}X$, $x \neq y$ exists $n \in {^*}\mathbb{Z}_{\infty}$ such that ${^*}d(({^*}f)^n(x), ({^*}f)^n(y)) > c$. It is clear that this stronger a priori version implies the nonstandard expansiveness, we will see that the reciprocal is true.\\
For the above theorem, the following formula is true:\\
$(\forall x \in X)(\forall y\in X)(x \neq y \rightarrow (\forall n \in \mathbb{N})(\exists i \in \mathbb{Z})(\abs{i} > n)\wedge(d(f^i(x),f^i(y)) > c)$. Then, for the transfer principle the formula\\
$(\forall x \in {^*}X)(\forall y\in {^*}X)(x \neq y \rightarrow (\forall n \in {^*}\mathbb{N})(\exists i \in {^*}\mathbb{Z})(\abs{i} > n)\wedge({^*}d(f^i(x),f^i(y)) > c)$ is also true. Then, for any  $x,y \in {^*}X$ and  $n \in {^*}\mathbb{N}_{\infty}$ exists $i \in \mathbb{Z}$ , with $\abs{i} > n$ such that $d(f^i(x),f^i(y))>c$, but  $n \in {^*}\mathbb{N}$, and then  $i \in {^*}\mathbb{Z}_{\infty}$. This proves that they are separated into infinite numbers.
\end{remark}

The next two propositions state that the nonstandard expansive dynamics verified the same properties of invariant by potence and conjugation as the expansive dynamics. These facts were proven by Groisman and da Silva \cite{jorgesamuel} in the context of the theory of ultrafilters. We give a proof with nonstandard methods.

\begin{prop} Let $X$ be a compact metric space. If $f : X \rightarrow X$ is nonstandard expansive, then for all $k \in \mathbb{N}$, $f^k$ is nonstandard expansive.
\end{prop}
\begin{proof}
Let $c$ be  the constant of expansivity of $f$. By uniform continuity, we know that  $\delta > 0$ exists such that if $d(x,y) < \delta$, then $d(f^i(x),f^i(y)) < c$ for all $i= 1, \ldots k$. Suppose that  $f^k$ is not nonstandard expansive, then $x_0,y_0 \in X$ exists such that for all $m\in {^*}\mathbb{Z}_{\infty}$,  $d(f^{mk}(x_0),f^{mk}(y_0)) < \delta$. Let $n \in {^*}\mathbb{Z}$, then for the Euclidean theorem (in a nonstandard version \footnote{ For all $a,b \in {^*}\mathbb{Z}$, $b >0$, exist $q, r \in {^*}\mathbb{Z}$ unique such that $a =bq + r$, with $0 \leq r < \abs{b}$. The proof is a simple application of transfer principle}) $m, i \in {^*}\mathbb{Z}$  exists, with $0\leq i < k$ such that $n=mk + i$, but if  $n \in {^*}\mathbb{Z}_{\infty}$, then $m \in {^*}\mathbb{Z}_{\infty}$, and then  $d(f^n(x_0), f^n(y_0)) = d(f^{mk + i}(x_0),f^{mk + i}(y_0)) = d(f^i(f^{mk}(x_0)),f^i(f^{mk}(y_0))) < c$, thus $f$ is not nonstandard expansive, which is absurd.
\end{proof}

\begin{prop} Let $X$, $Y$ be a metric compact spaces, and  $(X, f)$ and $(Y, g)$ be conjugated dynamics. If  $f : X \rightarrow X$ and if $f$ is nonstandard expansive, then  $g : Y \rightarrow Y$ is nonstandard expansive.
\end{prop}
\begin{proof}
Let  $\varphi : X \rightarrow Y$ be a homeomorphism and  $f : X \rightarrow X$ be nonstandard expansive with constant of expansivity  $\alpha$. We will prove that $g=\varphi.f.\varphi^{-1}$ is nonstandard expansive;  that is, we find $c\in \mathbb{R}$, $c >0$ such that for all  $a,b \in Y$, if  $d(g^n(a),g^n(b)) < c$ is satisfied for all $n\in {^*}\mathbb{Z}_{\infty}$, then  $a=b$.\\
For the transfer principle, if the following diagram is conmute
$\xymatrix{
 X \ar [r]^{f} \ar[d]_{\varphi}  & X \ar [d]^{\varphi} \\ 
 Y \ar [r]_{g} & Y
 }$
then the diagram
$\xymatrix{
 {^*}X \ar [r]^{{^*}f} \ar[d]_{{^*}\varphi}  &{^*} X \ar [d]^{{^*}\varphi} \\ 
 {^*}Y \ar [r]_{{^*}g} & {^*}Y
 }$
is conmute.
If  $X$ is compact and  $\varphi^{-1} : Y \rightarrow X$ is uniformly continuous, then $c\in \mathbb{R}$, $c>0$ exists such that if  $d(y,y') < c$, then $d(\varphi^{-1}(y),\varphi^{-1}(y')) < \alpha$. However, for the transfer principle ${^*}\varphi^{-1}$, we have verified that for all $y,y' \in {^*}Y$ if ${^*}d(y,y') < c$, then ${^*}d({^*}\varphi^{-1}(y),{^*}\varphi^{-1}(y')) < \alpha$. Suppose that for $a, b \in Y$ we have  ${^*}d({^*}g^n(a),{^*}g^n(b)) < c$ for all $n \in {^*}\mathbb{Z}_{\infty}$, then  ${^*}d({^*}\varphi^{-1}({^*}g^n(a),{^*}\varphi^{-1}({^*}g^n(b))) < \alpha$ for all  $n \in {^*}\mathbb{Z}_{\infty}$; but $a= \varphi(x)$ and  $b = \varphi (x')$ for some $x,x' \in X$, then  \\ ${^*}d{^*}(\varphi^{-1}({^*}g^n(\varphi(x)),{^*}\varphi^{-1}({^*}g^n(\varphi(x')))) < \alpha$; but $\varphi^{-1}({^*}g^n(\varphi(z)) = {^*}f^n(z)$ for all $z \in X$, and then  ${^*}d({^*}f^n(x),{^*}f^n(x')) <\alpha$ for all $n\in {^*}\mathbb{Z}_{\infty}$. If $f$ is  $\alpha$-expansive, then we have that $x=x'$, and then $a=b$.
\end{proof}

\begin{prop}  Let $X$ be a metric space and $f :X \rightarrow X$ be homeomorphisms, then 
$x,y \in X$ is positive asymptotic  if for all  $m \in {^*} \mathbb{N}_{\infty}$,  $f^{m}(x) \simeq f^m(y)$ \footnote{Let $X$  be a metric space and $x,y \in {^*}X$, we defined $x \simeq y$ if and only if ${^*}d(x,y)$ is an infinitesimal}.
\end{prop}
\begin{proof}
$\Rightarrow$: If $\lim\limits _{n \rightarrow + \infty} d(f^n(x), f^n(y))= 0$, then for  $\epsilon \in \mathbb{R}$, $\epsilon > 0$ exists $n_{\epsilon} \in \mathbb{N}$ such that  \\
 $(\forall m \in \mathbb{N})(m \geq \underline{n_{\epsilon}} \rightarrow d(f^m(x),f^m(y)) < \underline{\epsilon)}$. For the transfer principle, the following formula is true.\\
$(\forall m \in {^*}\mathbb{N})(m \geq n_{\epsilon} \rightarrow d(f^m(x),f^m(y)) < \underline{\epsilon})$, si $m \in {^*} \mathbb{N}_{\infty}$ , then  $m > n_{\epsilon}$ for all $\epsilon$; and then  $d(f^m(x),f^m(y)) < \epsilon$, for all $\epsilon$; and then $f^m(x) \simeq f^m(y)$.\\\\
$\Leftarrow$: Suppose that $\lim\limits _{n \rightarrow + \infty} d(f^n(x), f^n(y)) \neq 0$, then exists $\epsilon \in \mathbb{R}$, $\epsilon > 0$, such that for all $n\in \mathbb{N}$,  \\
$(\exists m\in \mathbb{N})( (m \geq n) \wedge d(f^m(x),f^m(y)) > \epsilon)$ is verified. Then by taking a choice function  $\psi : \mathbb{N} \rightarrow \mathbb{N}$, we have \\
 $(\forall n \in \mathbb{N})(\psi(n) > n) \wedge (d(f^{\psi(n)}(x), f^{\psi(n)}(y) > \underline{\epsilon})$ entonces for the transfer principle the formula $(\forall n \in {^*}\mathbb{N})({^*}\psi(n) > n) \wedge (d(f^{{^*}\psi(n)}(x), f^{{^*}\psi(n)}(y))>\underline{\epsilon})$ is true. If  $m \in {^*}\mathbb{N}_{\infty}$, then ${^*}\psi(m) \in {^*}\mathbb{N}_{\infty}$, with $d(f^{\psi(n)}(x), f^{\psi(n)}(y) >\epsilon$, and then   $f^m(x) \not\simeq f^m(y)$. Analogously, for  $x,y$ negative asymptotic.
\end{proof}

\begin{lema}\label{lema}
Let $X$ be a metric space, $a,b \in X$, $a', b' \in {^*}X$,  $r \in \mathbb{R}$, $r > 0$, then the following statements are true.
\begin{enumerate}
\item If ${^*}d(a',b') > r$, then $d(a,b) \geq r$.
\item If $d(a,b) > r$, then $st({^*}d(a',b'))  \geq r$. \footnote{Let $r \in {^*}\mathbb{R}$ be a nonstandard real number, then $st(r)$ is the nonstandard part of $r$. This is the unique $st(r) \in \mathbb{R}$ such that $r = st(r) + \xi$, where $\xi$ is an infinitesimal}
\end{enumerate}
\end{lema}
\begin{proof}
\begin{enumerate}
\item $r < {^*}d(a',b') \leq {^*}d(a'.a) + {^*}d(a,b) + {^*}d(b,b') = {^*}d(a',b') \leq {^*}d(a', a) + d(a,b) + {^*}d(b,b') = d(a,b) + \xi$, with $\xi$ infinitesimal, then $r < d(a,b) + \xi$, and then  $d(a,b) \geq r$, if it  is not then $d(a,b) < r$, and then $0 < r - d(a,b) < \xi$, which is absurd.
\item $r < d(a,b) = {^*}d(a,b) \leq {^*}d(a,a') + {^*}d(a',b') + {^*}d(b',b)$, but ${^*}d(a,a') \simeq 0$,  ${^*}d(b,b') \simeq 0$, then $r < {^*}d(a',b')  + \xi$, with $\xi \in {^*}\mathbb{R}$ is infinitesimal and not negative. If ${^*}d(a',b')$ is infinite, then it is proven. If this is not the case, then  $r=st(r) \leq st({^*}d(a',b') + \xi) = st({^*}d(a',b')) + st(\xi)= st({^*}d(a',b')) + 0 = st({^*}d(a',b'))$.
\end{enumerate}
\end{proof}

The following theorem was proven by Groisman and da Silva  for the freely expansive dynamics. We will present a proof that gives us a graphic intuition of the theorem.
\begin{teo} \label{doubly} Let $X$ be a compact metric space and  $f : X \rightarrow X$ be a homeomorphism, then $f$ is nonstandard expansive if and only if  $f$ is expansive and does not  have   doubly asymptotic points.
\end{teo}
\begin{proof}
$\Rightarrow$: If $f$ is nonstandard expansive, then for all  $x,y \in X$ exists $m \in {^*}\mathbb{Z}_{\infty}$, such that  $d(f^m(x),f^m(y)) > c$. Then for the above proposition,  $x,y$ are not asymptotic .\\\\
$\Leftarrow$: The idea of the proof is described in the following:

\vspace{0.3 cm}

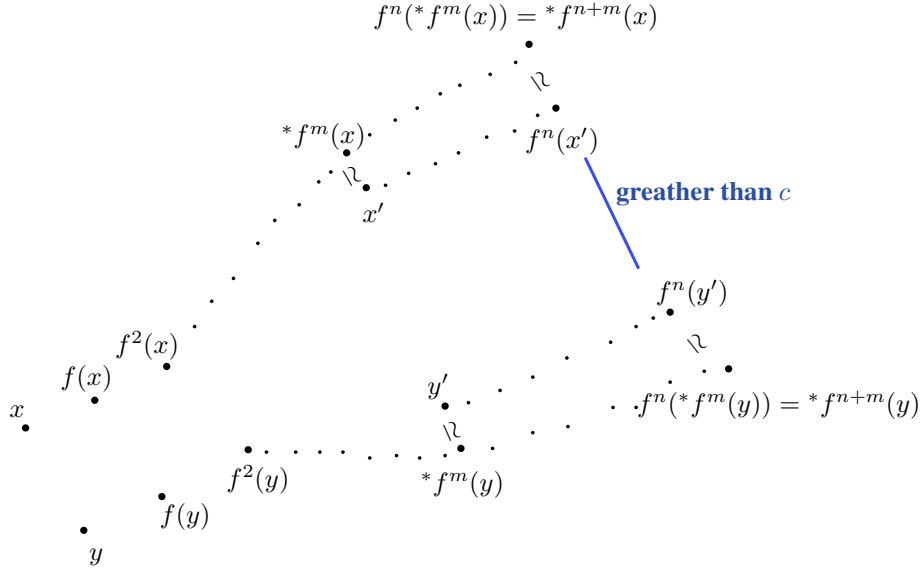
\psset{unit =0.8cm, linewidth=1.5\pslinewidth}
{
\begin{pspicture}(0,-4.684219)(16.169062,4.684219)
\definecolor{color817}{rgb}{0.1568627450980392,0.30196078431372547,0.6666666666666666}
\definecolor{color1236}{rgb}{0.21176470588235294,0.2980392156862745,0.996078431372549}
\psdots[dotsize=0.12](0.41,-2.3292189)
\psdots[dotsize=0.12](1.37,-4.0292187)
\psdots[dotsize=0.12](1.55,-1.8692187)
\psdots[dotsize=0.12](2.73,-1.3092188)
\psdots[dotsize=0.12](2.65,-3.4692187)
\psdots[dotsize=0.12](4.07,-2.6892188)
\psdots[dotsize=0.12](5.69,2.2307813)
\psdots[dotsize=0.12](7.57,-2.6692188)
\psdots[dotsize=0.12](6.01,1.6507813)
\psdots[dotsize=0.12](7.31,-1.9692187)
\psdots[dotsize=0.06](3.19,-0.64921874)
\psdots[dotsize=0.06](3.47,-0.22921875)
\psdots[dotsize=0.06](3.83,0.15078124)
\psdots[dotsize=0.06](4.23,0.73078126)
\psdots[dotsize=0.06](4.51,1.1107812)
\psdots[dotsize=0.06](4.91,1.3707813)
\psdots[dotsize=0.06](5.23,1.7507813)
\psdots[dotsize=0.06](5.43,1.9107813)
\psdots[dotsize=0.06](4.37,-2.7292187)
\psdots[dotsize=0.06](4.85,-2.7292187)
\psdots[dotsize=0.06](5.25,-2.7292187)
\psdots[dotsize=0.06](5.63,-2.7292187)
\psdots[dotsize=0.06](6.07,-2.8292189)
\psdots[dotsize=0.06](6.51,-2.7892187)
\psdots[dotsize=0.06](6.89,-2.8292189)
\psdots[dotsize=0.06](7.31,-2.7892187)
\psdots[dotsize=0.06](6.33,1.6907812)
\psdots[dotsize=0.06](6.71,1.8907813)
\psdots[dotsize=0.06](7.15,2.0507812)
\psdots[dotsize=0.06](7.53,2.2307813)
\psdots[dotsize=0.06](7.99,2.4907813)
\psdots[dotsize=0.06](8.45,2.5907812)
\psdots[dotsize=0.06](8.83,2.8507812)
\psdots[dotsize=0.06](7.69,-1.9092188)
\psdots[dotsize=0.06](8.21,-1.7092187)
\psdots[dotsize=0.06](8.75,-1.4892187)
\psdots[dotsize=0.06](9.31,-1.1692188)
\psdots[dotsize=0.06](10.03,-0.9092187)
\psdots[dotsize=0.06](10.45,-0.64921874)
\psdots[dotsize=0.06](10.73,-0.50921875)
\psdots[dotsize=0.12](9.13,2.9707813)
\psdots[dotsize=0.12](11.01,-0.40921876)
\psdots[dotsize=0.06](6.07,2.5307813)
\psdots[dotsize=0.06](6.49,2.7707813)
\psdots[dotsize=0.06](6.85,2.9707813)
\psdots[dotsize=0.06](7.23,3.2107813)
\psdots[dotsize=0.06](7.55,3.3507812)
\psdots[dotsize=0.06](8.05,3.5307813)
\psdots[dotsize=0.06](8.47,3.7307813)
\psdots[dotsize=0.06](8.09,-2.6692188)
\psdots[dotsize=0.06](8.75,-2.4092188)
\psdots[dotsize=0.06](9.33,-2.3092186)
\psdots[dotsize=0.06](10.07,-1.9892187)
\psdots[dotsize=0.06](10.45,-1.9892187)
\psdots[dotsize=0.06](11.01,-1.5692188)
\psdots[dotsize=0.06](11.57,-1.3892188)
\psdots[dotsize=0.12](11.97,-1.3492187)
\psdots[dotsize=0.12](8.69,4.0307813)
\usefont{T1}{ptm}{m}{n}
\rput(0.27453125,-2.0592186){$x$}
\usefont{T1}{ptm}{m}{n}
\rput(1.5745312,-4.459219){$y$}
\usefont{T1}{ptm}{m}{n}
\rput(1.4145312,-1.4992187){$f(x)$}
\usefont{T1}{ptm}{m}{n}
\rput(3.0345314,-3.7992187){$f(y)$}
\usefont{T1}{ptm}{m}{n}
\rput(2.4145312,-0.95921874){$f^{2}(x)$}
\usefont{T1}{ptm}{m}{n}
\rput(4.2345314,-3.1792188){$f^{2}(y)$}
\usefont{T1}{ptm}{m}{n}
\rput(5.3145313,2.5007813){${^*}f^{m}(x)$}
\usefont{T1}{ptm}{m}{n}
\rput(7.594531,-3.2592187){${^*}f^{m}(y)$}
\usefont{T1}{ptm}{m}{n}
\rput(6.134531,1.2407813){$x'$}
\usefont{T1}{ptm}{m}{n}
\rput(7.1945314,-1.6392188){$y'$}
\usefont{T1}{ptm}{m}{n}
\rput(9.234531,2.4207811){$f^{n}(x')$}
\usefont{T1}{ptm}{m}{n}
\rput(11.374531,-0.09921875){$f^n(y')$}
\usefont{T1}{ptm}{b}{n}
\rput(11.50,1.5607812){\color{color817}\text{ greather than $c$}}
\usefont{T1}{ptm}{m}{n}
\rput(12.80,-1.9392188){$f^n({^*}f^{m}(y)) = {^*}f^{n +m}(y)$}
\usefont{T1}{ptm}{m}{n}
\rput{-66.58771}(6.67999,5.3516893){\rput(7.384531,-2.3992188){$\simeq$}}
\usefont{T1}{ptm}{m}{n}
\rput{-60.430183}(1.3506374,5.939222){\rput(5.744531,1.8207812){$\simeq$}}
\usefont{T1}{ptm}{m}{n}
\rput{-55.8805}(1.0901445,8.835432){\rput(8.844531,3.4007812){$\simeq$}}
\usefont{T1}{ptm}{m}{n}
\rput{-51.991806}(5.1496534,8.660271){\rput(11.424531,-0.93921876){$\simeq$}}
\psline[linewidth=0.04cm,linecolor=color1236](9.61,2.1507812)(10.49,0.31078124)
\usefont{T1}{ptm}{m}{n}
\rput(8.474531,4.480781){$f^n({^*}f^{m}(x)) = {^*}f^{n +m}(x)$}
\end{pspicture} 
}

\vspace{0.3 cm}

If $x,y \in X$, are not  doubly asymptotic, then $m\in {^*}\mathbb{Z}_{\infty}$ exists such that $f^m(x) \not\simeq f^m(y)$. Then $\alpha \in \mathbb{R}$ exists such that  ${^*}d(f^m(x),f^m(y)) > \alpha$, and  $X$ is compact for the Robinson\footnote{Let $X$ be a metric space, then $X$ is compact if and only if  for all $y \in {^*}X$ exists $x \in X$, such that ${^*}d(x,y) \simeq 0$. See \cite{loebwolf}} theorem  $x', y' \in X$ exists such that $x' \simeq f^m(x)$ , $y' \simeq f^m(y)$, and if ${^*}d(f^m(x),f^m(y)) > \alpha$  then  $d(x',y') \geq \alpha$, in particular $x' \neq y'$. However, if  $f$ is expansive, then  $n \in \mathbb{Z}$ exists such that $d(f^n(x'),f^n(y')) >c$. However, if $f^n$ is continuous \footnote{Let $X$ be a metric space, $f : X \rightarrow X$ function and  $x_0 \in X$.  $f$ is continuous in  $x_0$ if and only if  for all $x \in {^*}X$, with $x \simeq x_0$, ${^*}f(x) \simeq f(x_0)$ is verified. See \cite{loebwolf}},then   ${^*}(f^n)(f^m(x)) \simeq f^n(x')$, ${^*}(f^n)(f^m(y))\simeq f^n(y')$. Then, for the lemma  \ref{lema} part 2,\\ $st(d(f^n(f^m(x)),f^n(f^m(y)))) \geq c$, but this implies that  \\ $d(f^n(f^m(x)),f^n(f^m(y))) > \frac{c}{2}$. If  $n \in \mathbb{Z}$ and  $m \in {^*}\mathbb{Z}_{\infty}$, then $n +m \in {^*}\mathbb{Z}_{\infty}$, and then $f$ is nonstandard expansive.
\end{proof}

An important example of expansive homeomorphis is the subshift. We know by \cite{KR} that every expansive homeomorphism on a totally disconnected compact metric space is a conjugate of a subshift. 
\begin{ej}
Let  $X \subset \Sigma^{\mathbb{Z}}$, where $\Sigma = \lbrace 1,2,3 \rbrace$. 
On $\Sigma^{\mathbb{Z}}$, we consider the distance $d(\alpha, \beta) = \sum_{i \in \mathbb{Z}} \abs{\alpha(i) - \beta(i)}2^{-\abs{i}}$ and the shift map $\sigma : \Sigma^{\mathbb{Z}} \rightarrow \Sigma^{\mathbb{Z}}$, $\sigma((\alpha(i))_{i\in \mathbb{Z}}) = (\alpha(i + 1))_{i \in \mathbb{Z}}$.
Let $0 < a < b < 1$ be rationally independent real numbers. 
Consider the interval exchange map $T: I \rightarrow I$, which is defined as follows:

Define the three intervals $I_1=[0,a)$, $I_2=[a,b)$ and 
$I_3=[b,1)$. Define the \emph{itinerary map}
$\iti: I\to\Sigma^{\mathbb{Z}}$ as $\iti(x) = (\alpha_k)_{k \in \mathbb{Z}}$ 
if $T^k(x)\in I_{\alpha_k}$ for all $k\in\Z$.
We obtain the following commutative diagram.

For $\alpha\in\Sigma^{\Z}$, we define $X_\alpha$ as the closure of the orbit $\{\sigma^n(\alpha):n\in \Z\}$. 
We say that $x\in I$ is \emph{regular} if 
$T^n(x)\notin \{0,a,b\}$ for all $n\in\Z$.\\
By - we know that this subshift does not have doubly asymptotic points, and then by the theorem \ref{doubly} is nonstandard expansive.
\end{ej}
The nonstandard analysis allows us to introduce a concept that is analogous to uniform expansiveness for expansive nonstandard dynamic.
\begin{defi}
Let  $X$ be a  metric space, and $f : X \rightarrow X$ be a homeomorphism. We say that $f$ is $c$-uniformly nonstandard expansive, and if for all $\epsilon > 0$ exists $n_{\epsilon} \in {^*}\mathbb{N}_{\infty}$ such that for all $x,y \in X$, and $d(x,y) > \epsilon$, then $d(f^i(x),f^i(y)) > c$, for some $i \in {^*}\mathbb{Z}_{\infty}$, $\abs{i} < n_{\epsilon}$.
\end{defi}

\begin{teo} \label{teounifexpnon}  Let $X$ be a metric compact space and $f : X \rightarrow X$ be homeomorphisms, then the following statements are equivalent.
\begin{enumerate}
\item $f$ is $c$- nonstandard expansive.\\
\item For all $\epsilon \in \mathbb{R}$ , $\epsilon > 0$, and  for all $n \in \mathbb{N}$, exists $m \in \mathbb{N}$, $m > n$ such that for all $x,y \in X$, if  $d(x,y)  > \epsilon$ then \\ $d(f^i(x),f^i(y)) > c$ for some  $i \in \mathbb{Z}$ with $n < \abs{i} < m$.\\
\item For all $\epsilon \in \mathbb{R}$, $\epsilon >0$ exists $n : \mathbb{N} \rightarrow \mathbb{N}$ succession is strictly monotonic such that $x,y \in X$, if $d(x,y) > \epsilon$ then for all $k \in\mathbb{N}$,   
$d(f^i(x),f^i(y)) > c$ is verified for some $i \in \mathbb{Z}$  with $n(k) < \abs{i} < n(k+1)$.\\
\item For all $\epsilon > 0$, $n_{\epsilon} \in {^*}\mathbb{N}_{\infty}$ exists  such that for all $x,y \in {^*}X$, if ${^*}d(x,y) > \epsilon$, then ${^*}d(f^i(x),f^i(y)) > c$, for some $i \in {^*}\mathbb{Z}_{\infty}$, $\abs{i} < n_{\epsilon}$.\\
\item $f$ is $c$-uniformly nonstandard expansive.
\end{enumerate}
\end{teo}
\begin{proof}
$1 \Rightarrow 2$ : Suppose that  $2$ is not true. Then  $\epsilon > 0$ y $n \in \mathbb{N}$ exists such that for all $m \in \mathbb{N}$, with $m > n$ exists $x_m,y_m \in X$ such that $d(x_m,y_m) > \epsilon$ and $d(f^i(x_m),f^i(y_m)) \leq c$ for all $i\in \mathbb{Z}$ with $n < \abs{i} < m$. If  $X$ is compact, then the taken  subsuccession is necessary. We can suppose that $(x_m)_{m\in \mathbb{N}}$ converge to $x$ and $(y_m)_{m\in \mathbb{N}}$ converge to $y$.\\
Let $k \in \mathbb{Z}$ such that  $\abs{k} > n$, be $\delta > 0$. If $f^k$ is continuous, then $f^k(x_m)$ converge to $f^k(x)$ and $f^k(y_m)$ converge to $f^k(y)$, then we can find  $m$ such that  $m > n$, $m > \abs{k}$,  $d(f^k(x_m),f^k(x)) < \frac{\delta}{2}$ y $d(f^k(y_m),f^k(y)) < \frac{\delta}{2}$, 
$d(f^k(x),f^k(y)) \leq d(f^k(x),f^k(x_m)) + d(f^k(x_m),f^k(y_m)) + d(f^k(y_m),f^k(y)) \leq 
\frac{\delta}{2} + c + \frac{\delta}{2} = c + \delta$, then  $d(f^k(x),f^k(y)) \leq c$ for all $k \in \mathbb{Z}$ with $\abs{k} > n$, and then $x,y$ they are separated at most by a finite number of points, and then $f$ is not nonstandard expansive, which is absurd.\\\\
$2 \Rightarrow 3$ : We will build a succession by recursion, take $n(1) = 1$ and  $n(2) > n(1)$ the $"m"$ of the statement $2$. Now suppose that we have defined \\ $n(1), n(2), \ldots , n(k)$, take $n(k+1) > n(k)$ the $"m"$ of the statement  $2$, then the statement is proven.\\\\
$3 \Rightarrow 4$ : Let  $\epsilon > 0$, $n : \mathbb{N} \rightarrow \mathbb{N}$ the succession of statement $3$, and a $l \in {^*}\mathbb{N}_{\infty}$. Let  $x,y \in X$ such that  $d(x,y) > \epsilon$ for the the numerable axiom of choice and the statement $3$, $i : \mathbb{N} \rightarrow \mathbb{Z}$ exists such that for all  $k \in \mathbb{N}$, $n(k) < \abs{i(k)} < n(k+1)$ y $d(f^{i(k)}(x), f^{i(k)}(y)) > c$. It is possible to extend the function $i$, such that  $i : X \times X \times \mathbb{N} \rightarrow \mathbb{Z}$. It is verified that if  $d(x,y) > \epsilon$, $i(x,y,k)$ be the old function and for the other case a constant arbitrary. Then, the following formula is true 
$(\forall x \in X)(\forall y \in X)( d(x,y) > \underline{\epsilon} \rightarrow (\forall k \in \mathbb{N})((n(k) < \abs{i(x,y,k)} < n(k+1))\wedge (d(f^{i(x,y,k)}(x),f^{i(x,y,k)}(x) > \underline{c})$. Then, for the transfer principle, the following formula is also true
$(\forall x \in {^*}X)(\forall y \in {^*}X)( {^*}d(x,y) > \underline{\epsilon} \rightarrow (\forall k \in {^*}\mathbb{N})(({^*}n(k) < \abs{{^*}i(x,y,k)} < {^*}n(k+1))\wedge ({^*}d({^*}f^{{^*}i(x,y,k)}(x),{^*}f^{{^*}i(x,y,k)}(x) > \underline{c})$
If  $l \in {^*}\mathbb{N}_{\infty}$  and  $n$ is strictly monotonic, then we have ${^*}n(l) \in {^*}\mathbb{N}_{\infty}$ and ${^*}n(l+1 ) \in {^*}\mathbb{N}_{\infty}$, then for all  $x, y \in {^*}\mathbb{X}$, ${^*}i(x,y,l) \in {^*}\mathbb{Z}_{\infty}$ is verified, and then the statement is proven.\\\\
$4 \Rightarrow 5$ : If $4$ is true, in particular $5$.\\\\
$5 \Rightarrow 1$ : Let  $x,y \in X$, $x\neq y$, take $\epsilon < d(x,y)$, because  $f$ is   $c$-uniformly nonstandard expansive, and then the statement is proven.
\end{proof}

\section{Questions and problems}

\begin{itemize}
\item We saw that  if  $f : X \rightarrow X$ is a homeomorphism with $X$ a compact metricspace, and if $f$ is nonstandard expansive, then $f$ is  uniformly nonstandard expansive. However, is this result true for a non-compact metric space? \\

\item In our work, all of the results are independent of the nonstandard model; that is, it does not depend on the ultrafilter with which it was built or the model. However, are there interesting dynamic properties that depend on the nonstandard model?\\

\item All the examples that we know of for homeomorphisms of compact spaces are also totally disconnected. However, are there dynamics of this type in environments without this topological constraint?
\end{itemize}

\renewcommand{\abstractname}{Acknowledgements}
\begin{abstract}
This work was carried out in the context of the Master's thesis \cite{rosa2019expansividad} by Facultad de Ciencias (Udelar). I wish to thank Professor Groisman for her guidance.

\end{abstract}


\begin{bibdiv}
	\begin{biblist}
\bib{KR}{article}{
	author={H.B. Keynes},
	author={J.B. Robertson},
	title={Generators for topological entropy and expansiveness},
	journal={Mathematical systems theory},
	volume={3},
	year={1969},
	pages={51--59}}

\bib{loebwolf}{book}{
	title={Nonstandard analysis for the working mathematician},
  author={Loeb, Peter A, and Wolff, Manfred PH},
  year={2000},
  publisher={Springer}
}





	

\bib{bryant1960expansive}{article}{
  title={On expansive homeomorphisms},
  author={Bryant, BF, and others},
  journal={Pacific J. Math},
  volume={10},
  pages={1163--1167},
  year={1960}
}


\bib{bryant1962expansive}{article}{
  title={Expansive self-homeomorphisms of a compact metric space},
  author={Bryant, BF},
  journal={The American Mathematical Monthly},
  volume={69},
  number={5},
  pages={386--391},
  year={1962},
  publisher={Taylor \& Francis}
}






\bib{jorgesamuel}{article}{
  title={Expansive dinamics in the sense of ultrafilters, unpublished notes},
  author={Jorge Groisman and Samuel G da Silva},
  journal={},
  volume={},
  number={},
  pages={},
  year={2016},
  publisher={Unpublished notes}
}



\bib{artigue2020subshift}{article}{
  title={Asymptotic Pairs for Interval Exchange Transformations},
  author={Alfonso Artigue,Ethan Akin, and Luis Ferrari},
  journal={arXiv preprint arXiv:2009.02592},
  year={2020}
}



\bib{bryant1966some}{article}{
  title={Some expansive homeomorphisms of the reals},
  author={Bryant, BF, and Coleman, DB},
  journal={The American Mathematical Monthly},
  volume={73},
  number={4},
  pages={370--373},
  year={1966},
  publisher={JSTOR}
}


\bib{rosa2019expansividad}{book}{
  title={Expansividad non-standard},
  author={Luis Ferrari},
  year={2019},
  publisher={Udelar. FC.}
}






\bib{utz1950unstable}{article}{
	title={Unstable homeomorphisms},
	author={W.R. Utz},
	journal={Proceedings of the American Mathematical Society},
	volume={1},
	number={6},
	pages={769--774},
	year={1950},
	publisher={JSTOR}
}		
	\end{biblist}
\end{bibdiv}
\end{document}